\documentclass[preprint,12pt]{elsarticle}

\usepackage{amssymb}

\usepackage{amsmath}
  \usepackage{paralist}
  \usepackage{graphics} 
  \usepackage{epsfig} 
 \usepackage[colorlinks=true]{hyperref}

\allowdisplaybreaks[4]

\usepackage{graphicx,amssymb,mathrsfs,latexsym,amsthm}
\usepackage{verbatim}

\newtheorem{theorem}{Theorem}
\newtheorem{lemma}[theorem]{Lemma}

\newtheorem{proposition}[theorem]{Proposition}
\newtheorem{remark}[theorem]{Remark}
\newtheorem{definition}[theorem]{Definition}

\begin{document}

\begin{frontmatter}



\title{Pre-image Variational Principle for Bundle Random Dynamical Systems}


\author[Shanghai,Nanjing]{Xianfeng Ma}
\ead{xianfengma@gmail.com}
\author[Nanjing,NJU]{Ercai Chen}
\ead{ecchen@njnu.edu.cn}
\address[Shanghai]{Department of Mathematics, East China University of Science and
Technology\\ Shanghai 200237, China}
\address[Nanjing]{School of Mathematics and Computer Science, Nanjing Normal
University\\Nanjing 210097, China}
\address[NJU]{Center of Nonlinear Science, Nanjing University\\Nanjing 210093,
China}

\begin{abstract}
The pre-image topological pressure is defined for bundle random
dynamical systems. A variational principle for it has also been
given.
\end{abstract}


\end{frontmatter}


\section{Introduction}
In deterministic dynamical systems, the thermodynamic formalism
based on the notions of pressure, of Gibbs and equilibrium states
plays a fundamental role in statistic mechanics, ergodic theory and
dimension  theory \cite{Ruelle1978,Bowen,Zelik2008,Pesin}. The
background for the study of equilibrium states is an appropriate
form of the variational principle. Its  first  version was
formulated by Ruelle \cite{Ruelle}.   In random dynamical systems
(RDS), the thermodynamic formalism is also important in the study of
chaotic properties of random transformations
\cite{Denker2008,Kifer1996,Bogen1995,Kifer1992}. The first version
of the variational principle for random transformations was given by
Ledrappier and Walters in the framework of the relativized ergodic
theory \cite{LedWal}, and it was extended by Bogensch{\"u}tz
\cite{Bogen} to random transformations acting on one place. Later
Kifer \cite{Kifer2001} gave the variational principle for random
bundle transformations.

In recent years, the pre-image structure of a map has also been
studied by many authors
\cite{Fiebig,Hurley,Langevin1,Langevin2,Cheng2005,Cheng2007,Nitecki}.
In deterministic dynamical systems, Fiebig \cite{Fiebig} studied the
relation between the classical topological entropy and the
dispersion of pre-images. Cheng and Newhouse \cite{Cheng2005}
introduced the notions of the pre-image entropies and obtained a
variational principle which is similar to the standard one. Zeng
\cite{Fanping} defined the notion of the pre-image pressure and
investigated its relationship with invariant measures. He also
established a variational principle for the pre-image pressure,
which was a generalization of Cheng's result for the pre-image
entropy. In random dynamical systems, Zhu \cite{Zhu} introduced the
analogous notions as that in the deterministic case and gave the
analogs of many known results for entropies, such as
Shannon-McMillan-Breiman Theorem, the Kolmogorov-Sinai Theorem, the
Abromov-Rokhlin formula and the variational principle.

In this paper, we present the notion of the pre-image topological
pressure and derive the corresponding variational principle for
bundle random dynamical systems. In fact, we formulate a random
variational principle between the pre-image topological pressure,
the pre-image measure-theoretic entropy and some functions of the
invariant measure. We also introduce a revised definition of the
random pre-image topological entropy without any additional
assumptions, while the original notion defined in \cite{Zhu} need a
strong measurability condition. All results in \cite{Zhu} still hold
for our new notion. For the probability space consisting of a single
point,  we establish the pre-image variational principle  on any
compact invariant subset for deterministic dynamical systems, which
is a generalization of Zeng's result \cite{Fanping} for the whole
space. The method we use is in the framework of Misiurewicz's
elegant proof \cite{Misiurewicz}. Kifer's method \cite{Kifer} and
Cheng's technique \cite{Cheng2005} are also adopted in the argument
of our theorem. In fact, our proof generalizes Kifer's proof of the
standard variational principle for random bundle transformations.

This paper is organized as follows. In Section \ref{Sec2}, we define
the pre-image measure-theoretic  entropy as a conditional entropy of
the induced skew product transformation  and give  another fiberwise
expression for bundle random dynamical systems. In Section
\ref{Sec3}, we define the pre-image topological pressure for bundle
random dynamical systems and give the power rule for this pressure.
In Section \ref{Sec4}, we state and prove the pre-image variational
principle.

\section{Pre-image measure-theoretic  entropy for bundle
RDS}\label{Sec2}

Let $(\Omega,\mathcal {F}, \mathbf{P})$ be a probability space
together with an invertible $\mathbf{P}$-preserving transformation
$\vartheta$, where $\mathcal{F}$ is  complete, countably generated
and separated points. Let $(X,d)$ be a compact metric space together
with the Borel $\sigma$-algebra $\mathcal{B}$. Let
$\mathcal{E}\subset \Omega \times X$ be measurable with respect to
the product $\sigma$-algebra $\mathcal{F}\times \mathcal{B}$ and the
fibers $\mathcal{E}_{\omega}=\{ x\in X: (\omega, x)\in
\mathcal{E}\}$, $\omega \in \Omega$ be compact. A continuous bundle
random dynamical system (RDS) $T$ over $(\Omega,\mathcal {F},
\mathbf{P}, \vartheta)$ is generated by map $T_{\omega}:
\mathcal{E_{\omega}}\rightarrow \mathcal{E_{\vartheta \omega}}$ with
iterates $T_{\omega}^{n}=T_{\vartheta^{n-1}\omega} \cdots
T_{\vartheta\omega}T_{\omega}$, $ n \geq 1$,  and
$T_{\omega}^{0}=id$, so that the map $(\omega, x) \rightarrow
T_{\omega}x$ is measurable and the map $x \rightarrow T_{\omega}x$
is continuous for $\mathbf{P}$-almost all (a.a) $\omega$. The map
$\Theta : \mathcal{E} \rightarrow \mathcal{E}$ defined by
$\Theta(\omega,x)=(\vartheta\omega, T_{\omega}x)$ is called the skew
product transformation.

Let $\mathcal {P}_{\mathbf{P}}(\mathcal{E}) = \{\mu \in \mathcal
{P}_{\mathbf{P}}(\Omega \times X) : \mu(\mathcal{E})=1\}$, where
$\mathcal {P}_{\mathbf{P}}(\Omega \times X)$ is the space of
probability measures on $\Omega \times X$ with the marginal
$\mathbf{P}$ on $\Omega$. Any $\mu \in \mathcal
{P}_{\mathbf{P}}(\mathcal{E})$ on $\mathcal{E}$ can be disintegrated
as $d \mu(\omega, x)=d \mu_{\omega}(x)\,d\mathbf{P}(\omega)$ (See
\cite{Dudley}), where $\mu_{\omega}$ are regular conditional
probabilities with respect to the $\sigma$-algebra $\mathcal
{F}_{\mathcal{E}}$ formed by all sets $(A\times X)\cap \mathcal{E}$
with $A \in \mathcal{F}$. Let $\mathcal
{M}_{\mathbf{P}}^1(\mathcal{E}, T)$ be the set of $\Theta$-invariant
measures $\mu \in\mathcal {P}_{\mathbf{P}}(\mathcal{E})$. $\mu$ is
$\Theta$-invariant if and only if the disintegrations $\mu_{\omega}$
of $\mu$ satisfy $T_{\omega}\mu_{\omega}=\mu_{\vartheta\omega}\,
\mathbf{P}$-a.s. \cite{Arnold}. Let $\mathcal {Q}=
\{\mathcal{Q}_{i}\}$ be a finite measurable partition of
$\mathcal{E}$, and $\mathcal {Q}(\omega)=
\{\mathcal{Q}_{i}(\omega)\}$, where $\mathcal{Q}_{i}(\omega)=\{x\in
\mathcal{E}_{\omega}: (\omega,x)\in \mathcal{Q}_{i}\}$ is a
partition of $\mathcal{E}_{\omega}$.
 For each $\omega \in \Omega$, let
$\mathcal{B}_{\omega}=\{B\cap \mathcal{E}_{\omega}: B\in
\mathcal{B}\}$ and $\mathcal{B}^-_{\omega}=\bigcap_{n\geq
0}(T^n_{\omega})^{-1}\mathcal{B}_{\vartheta^n\omega}$. Similarly,
let $(\mathcal{F}\times \mathcal{B})_{\mathcal{E}}=\{C\cap
\mathcal{E}:C \in \mathcal{F}\times \mathcal{B}\}$ and
$(\mathcal{F}\times \mathcal{B})^-_{\mathcal{E}}=\bigcap_{n\geq
0}\Theta^{-n}(\mathcal{F}\times \mathcal{B})_{\mathcal{E}}$.

For $\mu \in \mathcal {P}_{\mathbf{P}}(\mathcal{E})$, the
conditional entropy of  $\mathcal{Q}$ given by the $\sigma$-algebra
of $\mathcal {F}_{\mathcal{E}}\vee (\mathcal{F}\times
\mathcal{B})^-_{\mathcal{E}} \subset (\mathcal{F}\times
\mathcal{B})_{\mathcal{E}}$ is defined as usual (See Kifer
\cite{Kifer}) by
\begin{equation*}
\begin{split}
 &H_{\mu}(\mathcal{Q}\mid \mathcal {F}_{\mathcal{E}}\vee (\mathcal{F}\times
\mathcal{B})^-_{\mathcal{E}} )\\
 = &-\int \sum_i \mu(\mathcal{Q}_i \mid
\mathcal{F}_{\mathcal{E}}\vee (\mathcal{F}\times
\mathcal{B})^-_{\mathcal{E}} ) \log \mu(\mathcal{Q}_i \mid
\mathcal{F}_{\mathcal{E}}\vee (\mathcal{F}\times
\mathcal{B})^-_{\mathcal{E}} )\,d\mu .
\end{split}
\end{equation*}
The pre-image measure-theoretic entropy
$h_{\text{pre},\,\mu}^{(r)}(T)$ of bundle RDS $T$ with respect to
$\mu$ is defined by the formula
\begin{equation*}
\label{entropy}
 h_{\text{pre},\,\mu}^{(r)}(T)=\sup_{\mathcal{Q}}h_{\text{pre},\,\mu}^{(r)}(T,
\mathcal{Q}),
\end{equation*}
where
\begin{equation*}
h_{\text{pre},\,\mu}^{(r)}(T, \mathcal{Q})=\lim_{n\rightarrow
\infty}\frac{1}{n}H_{\mu}\bigl(
\bigvee_{i=0}^{n-1}\Theta^{-i}\mathcal{Q}|\mathcal{F}_{\mathcal{E}}\vee(\mathcal{F}\times
\mathcal{B})^-_{\mathcal{E}} \bigr)
\end{equation*}
and the supremum is taken over all finite or countable  measurable
partitions $\mathcal{Q}=\{\mathcal{Q}_i\}$ of $\mathcal{E}$ with
finite conditional entropy $H_{\mu}(\mathcal{Q}\mid
\mathcal{F}_{\mathcal{E}})< \infty$. The existence of the limit
follows from the formula
$\Theta^{-1}(\mathcal{F}_{\mathcal{E}}\vee(\mathcal{F}\times
\mathcal{B})^-_{\mathcal{E}}) \subset
\mathcal{F}_{\mathcal{E}}\vee(\mathcal{F}\times
\mathcal{B})^-_{\mathcal{E}}$
 and the standard subadditive argument (cf. Kifer \cite[Theorem
 II.1.1]{Kifer}).  The resulting entropy remains the same by taking
the supremum  only over partitions $\mathcal{Q}$ of $\mathcal{E}$
into sets $Q_i$ of the form $Q_i=(\Omega\times P_i)\cap\mathcal{E}$,
where $\mathcal{P}=\{P_i\}$ is a partition of $X$ into measurable
sets, so that $Q_i(\omega)=P_i\cap\mathcal{E}_{\omega}$ (See
\cite{Bogen,Bogenthesis,Kifer} for detail).

Compared with Zhu \cite{Zhu}, we define the pre-image
measure-theoretic entropy on the measurable subset $\mathcal{E}$
instead of on the whole space $\Omega\times X$. Moreover, if
$\mathcal{E}=\Omega\times X$, then the above definition is just the
measure-theoretic pre-image entropy in \cite{Zhu}. In this sense,
the definition is a generalization of Zhu's.

In \cite{Zhu}, Zhu gave another fiberwise expression for his defined
 measure-theoretic pre-image entropy, which can be seen as a
generalization of Kifer's discussion on the standard
measure-theoretic entropy for random bundle transformations
\cite{Kifer2001}. In a similar way, we can give a fiberwise
expression for the above definition. For completeness of this paper,
we state this proposition and give the proof.

\begin{proposition}
Let $\mathcal{Q}$ be a finite measurable partition of $\mathcal{E}$.
Then
 \begin{equation}\label{eq1}
 h_{\rm{pre},\,\mu}^{(r)}(T,
\mathcal{Q})=\lim_{n\rightarrow\infty}\frac{1}{n}\int
H_{\mu_{\omega}}\bigl(
\bigvee_{i=0}^{n-1}(T^i_{\omega})^{-1}\mathcal{Q}(\vartheta^i\omega)|\mathcal{B}^-_{\omega}\bigr)
\,d\mathbf{P}(\omega).
\end{equation}
\end{proposition}

\begin{proof}
Note that for any $f\in \mathbf{L}^1(\Omega\times X,\mu)$ and
$\bigcup_{\omega\in F}\{\omega\}\times B_{\omega} \in
\mathcal{F}_{\mathcal{E}}\vee (\mathcal{F}\times
\mathcal{B})^-_{\mathcal{E}} $ where $F\in
\Pr_{\Omega}\mathcal{F}_{\mathcal{E}}$ and $B_{\omega} \in
\mathcal{B}^-_{\omega} $, $\Pr_{\Omega}\mathcal{F}_{\mathcal{E}}$ is
the projection of $\mathcal{F}_{\mathcal{E}}$ into $\Omega$, we have
\begin{align*}
&\int_F\int_{B_{\omega}}E(f|\mathcal{F}_{\mathcal{E}}\vee
(\mathcal{F}\times \mathcal{B})^-_{\mathcal{E}}
)(\omega,x)\,d\mu_{\omega}(x)\,d\mathbf{P}(\omega)\\
=&\int_{\bigcup_{\omega\in F}\{\omega\}\times B_{\omega}
}f(\omega,x)\,d\mu(\omega,x)\\
=&\int_F\int_{B_{\omega}}f_{\omega}(x)
\,d\mu_{\omega}(x)\,d\mathbf{P}(\omega)\\
=&\int_F\int_{B_{\omega}} E(f_{\omega}|\mathcal{B}^-_{\omega})
\,d\mu_{\omega}(x)\,d\mathbf{P}(\omega)
\end{align*}
where $f_{\omega}(x)=f(\omega,x)$. Therefore,
$$
E(f|\mathcal{F}_{\mathcal{E}}\vee (\mathcal{F}\times
\mathcal{B})^-_{\mathcal{E}}
)(\omega,x)=E(f_{\omega}|\mathcal{B}_{\omega}^-)(x) \quad \quad
\mu-a.e.
$$
Hence for any finite measurable partition $\mathcal{Q}$ of
$\mathcal{E}$, we have
$$
I_{\mu}(\mathcal{Q}|\mathcal{F}_{\mathcal{E}}\vee (\mathcal{F}\times
\mathcal{B})^-_{\mathcal{E}})(\omega,x)=
I_{\mu_{\omega}}(\mathcal{Q}({\omega})|\mathcal{B}_{\omega}^-)(x)
\quad \quad \mu-a.e.
$$
where $I_{\cdot}(\cdot|\cdot)$ denotes the standard conditional
information function. Thus
$$
H_{\mu}(\mathcal{Q}|\mathcal{F}_{\mathcal{E}}\vee (\mathcal{F}\times
\mathcal{B})^-_{\mathcal{E}})=\int
I_{\mu_{\omega}}(\mathcal{Q}({\omega})|
\mathcal{B}_{\omega}^-)\,\mathbf{P}(\omega) .
$$
Since $(\Theta^{-i}\mathcal{Q})(\omega)=
(T_{\omega}^i)^{-1}\mathcal{Q}(\vartheta^i\omega)$ for any $i\in
\mathbb{N}$, then
$$
H_{\mu}(\bigvee_{i=0}^{n-1}\Theta^{-i}\mathcal{Q}|\mathcal{F}_{\mathcal{E}}\vee
(\mathcal{F}\times \mathcal{B})^-_{\mathcal{E}})= \int
H_{\mu_{\omega}}(\bigvee_{i=0}^{n-1}
(T_{\omega}^i)^{-1}\mathcal{Q}(\vartheta^i\omega)|\mathcal{B}_{\omega}^-)
\,d\mathbf{P}(\omega).
$$
Dividing by $n$ and letting $n\rightarrow \infty$, we obtain the
equality \eqref{eq1}.
\end{proof}

Moreover, if we use Zhu's methods  and  restrict the whole space
$\Omega\times X$ to the measurable subset $\mathcal{E}$ in
\cite{Zhu}, then all results with respect to the pre-image
measure-theoretic entropy defined by Zhu also hold for the above
definition. Then, we can use those results directly without giving
any proof whenever we consider the pre-image measure-theoretic
entropy.

\section{Pre-image topological pressure for bundle
RDS}\label{Sec3}

Let $X_{\mathcal{E}}=\{x\in X: (\omega,x)\in \mathcal{E},
\omega\in\Omega\}$. For each $x\in X_{\mathcal{E}}$, by the
measurability of bundle RDS $T$, $\mathcal{E}(x)=\{(\omega, y):
\omega\in \Omega, y\in T^{-1}_{\omega}x\}$ is measurable with
respect to the product $\sigma$-algebra $\mathcal{F}\times
\mathcal{B} $. For each $k\in \mathbb{N}$, let
$\mathcal{E}(x,k)=\{(\omega, y): \omega\in \Omega, y\in
(T^k_{\omega})^{-1}x\}$. By the continuity of bundle  RDS $T$ and
\cite[Theorem III.30]{Castaing}, it is not hard to see that
$\mathcal{E}(x,k)$ is also measurable. Since for each $k\in
\mathbb{N}$, $(T^k_{\omega})^{-1}x$ is compact in
$\mathcal{E}_{\omega}$, then the mapping $\omega:\rightarrow
(T^k_{\omega})^{-1}x$ is measurable (See\cite[Chaper III]{Castaing})
with respect to the Borel $\sigma$-algebra induced by the Hausdorff
topology  on the space $\mathcal{K}(X)$, and  the distance function
$d(z,(T^k_{\omega})^{-1}x)$ is measurable in $\omega \in \Omega$ for
each $z\in X$.

For each $n\in\mathbb{N}$, a family of metrics $d_n^{\omega}$ on
$\mathcal{E}_{\omega}$ is defined as
$$d_n^{\omega}(y,z)=\max_{0 \leq i < n}(d(T_{\omega}^i y,
T_{\omega}^i z)), \quad  y,z \in \mathcal{E}_{\omega}.$$ It is not
hard to see that for each $k\in \mathbb{N}$ and $x\in
X_{\mathcal{E}}$, the set $\mathcal{E}^{(2)}(x,k)=\{(\omega,y,z):
y,z \in (T^k_{\omega})^{-1}x\}$ belongs to the product
$\sigma$-algebra $\mathcal{F}\times \mathcal{B}^2$ (See
\cite[Proporsition III.13]{Castaing}). Since for each $m\in
\mathbb{N}$, $\epsilon>0$ and a real number $a$ the set
$\{(\omega,y,z)\in \mathcal{E}^{(2)}(x,k): d(T^m_{\omega}y,
T^m_{\omega}z)\leq a\epsilon\}$ is measurable with respect to this
product $\sigma$-algebra, then $d_n^{\omega}(y,z)$ depends
measurably on $(\omega,y,z)\in \mathcal{E}^{(2)}(x,k)$.

For each $n\in\mathbb{N}$  and $\epsilon >0$, a set $F\subset
\mathcal{E}_{\omega}$ is said to be $(\omega, n,\epsilon
)$-separated if for any $y,z \in F$, $y\neq z$ implies
$d_n^{\omega}(y,z)>\epsilon$. Similarly, for a compact subset
$K\subset \mathcal{E}_{\omega}$, $F\subset K$ is said to be
$(\omega, n,\epsilon )$-separated for $K$ if for any $y,z\in F$,
$y\neq z$ implies $d_n^{\omega}(y,z)>\epsilon$.

Due to the compactness, there exists a smallest natural number
$s_n(\omega,\epsilon)$ such that $\text{card}(F)\leq
s_n(\omega,\epsilon)<\infty$ for every $(\omega, n,\epsilon
)$-separated $F$. Moreover, there always exists a maximal $(\omega,
n,\epsilon )$-separated set $F$ in the sense that for every $y\in
\mathcal{E}_{\omega}$ with $y\not\in F $ the set $F \cup \{y\}$ is
not $(\omega, n,\epsilon )$-separated anymore. In particular, this
is also true for any compact subset $K$ of $\mathcal{E}_{\omega}$.
Let $s_n(\omega,\epsilon,K)$ be the smallest natural number such
that $\text{card}(F)\leq s_n(\omega,\epsilon,K)<\infty$ for every
$(\omega, n,\epsilon )$-separated set $F$ of $K$.

For each function $f$ on $\mathcal{E}$, which is measurable in
$(\omega, x)$ and continuous in $x \in \mathcal{E}_{\omega}$, let
\begin{equation*}
\| f \| =\int \|  f(\omega) \|_{\infty} \, d\mathbf{P}, \quad
\text{where} \quad \|  f(\omega) \|_{\infty} = \sup_{x\in
\mathcal{E}_{\omega}}\mid f(\omega,x) \mid.
\end{equation*}
Let $\mathbf{L}_{\mathcal{E}}^1 (\Omega, \mathcal{C}(X))$ be  the
space of such functions $f$ with $\|f \| < \infty$. If we identify
$f$ and $g$ for $f, g \in
 \mathbf{L}_{\mathcal{E}}^1 (\Omega, \mathcal{C}(X))$ with $\| f-g\| =0$, then
$\mathbf{L}_{\mathcal{E}}^1 (\Omega, \mathcal{C}(X))$ is a Banach
space with the norm $\| \cdot \|$.

For any $\delta>0$, let
$$
\kappa_{\delta}^{(f)}(\omega)=\sup\{|f(\omega,x)-f(\omega,y)|:x,y\in
\mathcal{E}_{\omega}, d(x,y)\leq \delta\}.
$$

For $f \in \mathbf{L}_{\mathcal{E}}^1 (\Omega, \mathcal{C}(X))$,
$k,n\in \mathbb{N}$ with $k\geq n$, $\epsilon>0$, $x\in
X_{\mathcal{E}}$ with $(T_{\omega}^k)^{-1}x\neq \emptyset$,  and an
$(\omega, n, \epsilon)$-separated set $E$ of $(T^k_{\omega})^{-1}x$
such that $E\subset (T^k_{\omega})^{-1}x\subset
\mathcal{E}_{\omega}$, set
\begin{equation*}
S_nf(\omega,y)=\sum_{i=0}^{n-1}f(\vartheta^i\omega, T^i_{\omega}y)=
\sum_{i=0}^{n-1}f\circ\Theta^i(\omega,y),
\end{equation*}
and denote
\begin{equation*}
P_{\text{pre},\,n,\,
\omega}(T,f,\epsilon,(T^k_{\omega})^{-1}x)=\sup_{E}\sum_{y\in E}\exp
S_nf (\omega,y) ,
\end{equation*}
where the supremum is taken over all $(\omega,n,
\epsilon)$-separated sets of $(T^k_{\omega})^{-1}x$ in
$\mathcal{E}_{\omega}$.  Based on the foregoing analysis, clearly,
any $(\omega,n,\epsilon)$-separated set can be completed to a
maximal one. Then, the supremum can be taken only over all maximal
$(\omega, n ,\epsilon)$-separated sets.
 For $(T_{\omega}^k)^{-1}x=
\emptyset$, let $P_{\text{pre},\,n,\,
\omega}(T,f,\epsilon,(T^k_{\omega})^{-1}x)=0$ for all $n$ and
$\epsilon$, then the function $P_{\text{pre},\,n,\,
\omega}(T,f,\epsilon,(T^k_{\omega})^{-1}x)$ is well-defined. In this
paper, we always assume that for each $k\in\mathbb{N}$ and $x\in
X_{\mathcal{E}}$, $(T_{\omega}^k)^{-1}x\neq \emptyset$.
Alternatively, we can  also assume that for each $\omega$ the
mapping $T_{\omega}$ is  surjective.

The following auxiliary result, which relies on Kifer's work
\cite{Kifer2001} and restricts his result to the family of compact
subsets  $(T_{\omega}^k)^{-1}x$ of $\mathcal{E}_{\omega}$ for
nonrandom positive number $\epsilon$, provides the basic properties
of measurability needed in what follows. We make a little adjustment
to Kifer's proof for the purpose of defining the pre-image
topological pressure.

\begin{lemma}\label{lem}
For each $x\in X_{\mathcal{E}}$,  $k, n\in \mathbb{N}$ with $k\geq
n$, and a  nonrandom small positive number $\epsilon$, the function
$P_{{\rm pre},\,n,\, \omega}(T,f,\epsilon,(T^k_{\omega})^{-1}x)$ is
measurable in $\omega$, and for any $\delta >0$, there exists a
family of maximal $(\omega,n,\epsilon)$-separated sets
$G_{\omega}\subset(T^k_{\omega})^{-1}x\subset \mathcal{E}_{\omega}$
satisfying
\begin{equation}\label{inseq6}
\sum_{y\in G_{\omega}}\exp S_nf (\omega,y)\geq (1-\delta)P_{{\rm
pre},\,n,\, \omega}(T,f,\epsilon,(T^k_{\omega})^{-1}x)
\end{equation}
and depending measurably on $\omega$ in the sense that
$G=\{(\omega,x):x\in G_{\omega}\}\in \mathcal{F}\times \mathcal{B}$,
which also means that the mapping $\omega\rightarrow G_{\omega}$ is
measurable with respect to the Borel $\sigma$-algebra induced by the
Hausdorff topology on the $\mathcal{K}(X)$ of compact subsets of
$X$. In particular, the supremum in the definition of $P_{{\rm
pre},\,n,\, \omega}(T,f,\epsilon,(T^k_{\omega})^{-1}x)$ can be taken
only over families of $(\omega,n, \epsilon)$-separated sets, which
are measurable in $\omega$.
\end{lemma}

\begin{proof}
Fix $x\in X_{\mathcal{E}}$. For $q, n \in \mathbb{N}_+$, let
\begin{align*}
D_q&=\{(\omega,x_1,\ldots,x_q):\omega\in \omega, x_i\in
(T^k_{\omega})^{-1}x, \forall i\},\\
E_q^n&=\{(\omega,x_1,\ldots,x_q)\in D_q:
d_n^{\omega}(x_i,x_j)>\epsilon, \forall i\neq j \},\\
E_q^{n,l}&=\{(\omega,x_1,\ldots,x_q)\in D_q:
d_n^{\omega}(x_i,x_j)\geq \epsilon + 1/l, \forall i\neq j \},\\
E_q^n(\omega)&=\{(x_1,\ldots,x_q): (\omega,x_1,\ldots,x_q)\in
E_q^n\}\\
E_q^{n,l}(\omega)&=\{(x_1,\ldots,x_q): (\omega,x_1,\ldots,x_q)\in
E_q^{n,l}\}.
\end{align*}
Observe that $D_q \in \mathcal{F}\times \mathcal{B}^q$, where
$\mathcal{B}^q$ is the product $\sigma$-algebra on the product of
$q$ copies of $X$. This follows from \cite[Theorem III.30]{Castaing}
since
$$
d_q((x_1,\ldots,x_q),(y_1,\ldots,y_q))=\sum_{i=1}^qd(x_i,y_i)
$$
is the distance function on $X^q$; and if
$\mathcal{E}_{\omega}^{(q)}(x,k)$ denotes the product of $q$ copies
of $\mathcal{E}_{\omega}(x,k)=\{y: (\omega,y)\in
\mathcal{E}(x,k)\}$, then
$$
d_q((x_1,\ldots,x_q),\mathcal{E}_{\omega}^{(q)}(x,k)
)=\sum_{i=1}^qd(x_i,\mathcal{E}_{\omega}(x,k))
$$
is measurable in $\omega$ for each $(x_1,\ldots,x_q) \in X^q$. Next
we define $q(q-1)/2$ measurable functions $\psi_{ij}, 1\leq i<j\leq
q$, on $D_q$ by
$\psi_{ij}(\omega,x_1,\ldots,x_q)=d_n^{\omega}(x_i,x_j)$. Then
$$
E_q^{n,l}=\bigcap_{1\leq i<j\leq
q}\psi_{ij}^{-1}[\epsilon+1/l,\infty) \in \mathcal{F}\times
\mathcal{B}^q.
$$
By the continuity of the RDS $T$, each $E_q^{n,l}(\omega)$ is a
closed subset of $\mathcal{E}_{\omega}^{(q)}(x,k)$, thus it is
compact. Clearly, $E_q^{n,l} \uparrow  E_q^n$ and
$E_q^{n,l}(\omega)\uparrow E_q^n(\omega)$ as $l \rightarrow \infty$.
In particular, $E_q^n\in \mathcal{F}\times \mathcal{B}^q$.

Let $s_n(\omega,\epsilon)$ be the largest cardinality of all
$(\omega,n, \epsilon)$-separated set in $\mathcal{E}_{\omega}(x,k)$
and $t_{n,l}(\omega,\epsilon)=\max\{q:E_q^{n,l}(\omega)\neq
\emptyset \}$. Then by \cite[Theorem III.23]{Castaing}, we have
$$
\{\omega:t_{n,l}(\omega,\epsilon)\geq q\}=\{\omega
:E_q^{n,l}(\omega)\neq \emptyset\}={\rm Pr}_{\Omega}E_q^{n,l}\in
\mathcal{F},
$$
where ${\rm Pr}_{\Omega}$ is the projection of $\Omega\times X^q$ to
$\Omega$. Thus $t_{n,l}(\omega,\epsilon)$ is measurable in $\omega$.
Now we get
\begin{align*}
\{\omega:s_n(\omega,\epsilon)\geq q\}=\{\omega:E_q^n(\omega)\neq
\emptyset\}=\bigcup_{m=1}^{\infty}\bigcap_{l=m}^{\infty}\{\omega
:E_q^{n,l}(\omega)\neq \emptyset\}\in \mathcal{F}.
\end{align*}
Then $s_n(\omega,\epsilon)$ is measurable in $\omega$ as well. Since
$s_n(\omega,\epsilon)\geq t_{n,l}(\omega,\epsilon)$ for all $l\geq
1$, then $t_{n,l}(\omega,\epsilon)\uparrow s_n(\omega,\epsilon)$ as
$l\rightarrow \infty$, thus $t_{n,l}(\omega,\epsilon)=
s_n(\omega,\epsilon)$ for all $l$ large enough (depending on $n$ and
$\omega$). By \cite[Lemma III.39]{Castaing}, each function
$$
g_{q,l}=\sup\bigl\{\sum_{i=1}^q\exp(S_nf(\omega,x_i)):(x_1,\ldots,x_q)\in
E_q^{n,l}(\omega) \bigr\}
$$
is measurable. Thus the functions $g_q=\sup_{l\geq1}g_{q,l}$ and
$P_{\text{pre},\,n,\,
\omega}(T,f,\epsilon,(T^k_{\omega})^{-1}x)=\max_{1\leq q \leq
s_n(\omega,\epsilon)}g_q(\omega)$ are both measurable.

For each constant $\delta>0$,  the set
$$
F_{q,\delta}^n=\bigl\{(\omega,x_1,\ldots, x_q)\in D_q:
\sum_{i=1}^q\exp(S_nf(\omega,x_i))\geq (1-\delta)g_q(\omega) \bigr\}
$$
belongs to $\mathcal{F}\times\mathcal{B}^q $. Hence
$$
G_{q,\delta}^n=F_{q,\delta}^n\cap E_q^n\in
\mathcal{F}\times\mathcal{B}^q \quad \text{and }\quad
G_{q,\delta}^{n,l}= F_{q,\delta}^n\cap E_q^{n,l}\in
\mathcal{F}\times\mathcal{B}^q.
$$
Let
\begin{align*}
G_{q,\delta}^n(\omega)&=\{(x_1,\ldots,x_q):
(\omega,x_1,\ldots,x_q)\in G_{q,\delta}^n\},\\
G_{q,\delta}^{n,l}(\omega)&=\{(x_1,\ldots,x_q):
(\omega,x_1,\ldots,x_q)\in G_{q,\delta}^{n,l}\}.
\end{align*}
Observe that $G_{q,\delta}^{n,l}(\omega)$ are compact sets and
$G_{q,\delta}^{n,l}(\omega)\uparrow G_{q,\delta}^n(\omega)$ as
$l\rightarrow \infty$. The sets
$$
\widetilde{\Omega}_{q,l}=\{\omega:
t_{n,l}(\omega,\epsilon)=s_n(\omega,\epsilon)=q\}\cap \{\omega:
G_{q,\delta}^{n,l}(\omega)\neq \emptyset\}
$$
are, clearly, measurable, and the sets
$\Omega_{q,l}=\widetilde{\Omega}_{q,l}\backslash
\widetilde{\Omega}_{q,l-1}, l=1,2,\ldots$, with
$\widetilde{\Omega}_{q,0}=\emptyset$ are measurable, disjoint and
$\bigcup_{q,l\geq 1}\Omega_{q,l}=\Omega$. Thus (See \cite[Theorem
III.30]{Castaing}),  the multifunction $\Psi_{q,l,\delta}$ defined
by $\Psi_{q,l,\delta}(\omega)=G_{q,\delta}^{n,l}(\omega)$ for
$\omega \in \Omega_{q,l}$ is measurable, and it admits a measurable
selection $\sigma_{q,l,\delta}$ which is measurable map
$\sigma_{q,l,\delta}:\Omega_{q,l}\rightarrow X^q$ such that
$\sigma_{q,l,\delta}(\omega)\in G_{q,\delta}^{n,l}(\omega)$ for all
$\omega \in \Omega_{q,l}$. Let $\zeta_q$ be the multifunction from
$X^q$ to $q$-point subsets of $X$ defined by
$\zeta_q(x_1,\ldots,x_q)=\{x_1,\ldots,x_q\}\subset X$. Then
$\zeta_q\circ\sigma_{q,l,\delta}$ is a multifunction assigning to
each $\omega \in \Omega_{q,l}$ a maximal
$(\omega,n,\epsilon)$-separated set $G_{\omega}\subset
(T^k_{\omega})^{-1}x $ in $ \mathcal{E}_{\omega}$ for which
\eqref{inseq6} holds true.

For any open set $U \subset X$, let $V_U^q(i)=\{(x_1,\ldots,x_q)\in
X^q: x_i\in U\}$, which is an open set of $X^q$. Then, clearly,
$$
\{\omega\in \Omega_{q,l}:
\zeta_q\circ\sigma_{q,l,\delta}(\omega)\cap U\neq \emptyset\} =
\bigcup_{i=1}^q \sigma_{q,l,\delta}^{-1}V_U^q(i)\in \mathcal{F}.
$$
Now we define the random variable $m_q$ by $m_q(\omega)=l$ for all
$\omega\in \Omega_{q,l}$. Let $\Phi_{\delta}(\omega)=
\zeta_{s_n(\omega,\epsilon)} \circ
\sigma_{s_n(\omega,\epsilon),m_q(\omega),\delta}(\omega)$; then
$$
\{\omega:\Phi_{\delta}(\omega)\cap U\neq
\emptyset\}=\bigcup_{q,l=1}^{\infty}\{\omega\in \Omega_{q,l}:
\zeta_q\circ\sigma_{q,l,\delta}(\omega)\cap U\neq \emptyset\}\in
\mathcal{F}.
$$
Hence $\Phi_{\delta}$ is a measurable multifunction which assigns to
each $\omega\in \Omega$ a maximal $(\omega,n,\epsilon)$-separated
set $G_{\omega}$ for which \eqref{inseq6} holds  and Lemma \ref{lem}
follows since $\delta>0$ is arbitrary.
\end{proof}

In view of this assertion, for each $f\in
\mathbf{L}^1_{\mathcal{E}}(\Omega, \mathcal{C}(X))$ and any positive
number $\epsilon$ we can introduce the function
\begin{equation}\label{predefn}
P_{\text{pre}}(T,f,\epsilon)=\limsup_{n\rightarrow
\infty}\frac{1}{n} \sup _{k\geq n}\sup_{x\in X_\mathcal{E}}\int \log
P_{\text{pre},\,n,\,
\omega}(T,f,\epsilon,(T^k_{\omega})^{-1}x)\,d\mathbf{P}(\omega).
\end{equation}
Note that though we set $x\in X_\mathcal{E}$, only those points such
that $x\in \mathcal{E}_{\vartheta^k\omega}$ act in the function
$P_{\text{pre}}(T,f,\epsilon)$.

\begin{definition}\label{defn1}
The pre-image topological pressure of a function $f\in
\mathbf{L}^1_{\mathcal{E}}(\Omega, \mathcal{C}(X))$ for bundle RDS
$T$ is the map
\begin{equation*}
P_{\text{pre}}(T,\cdot): \mathbf{L}^1_{\mathcal{E}}(\Omega,
\mathcal{C}(X)) \rightarrow \mathbb{R}\cup \{\infty\},
\,\,\text{where}\,\,
 P_{\text{pre}}(T,f)=\lim_{\epsilon\rightarrow
0}P_{\text{pre}}(T,f,\epsilon).
\end{equation*}
The limit exists since $P_{\text{pre}}(T,f,\epsilon)$ is monotone in
$\epsilon$ and, in fact, $\lim_{\epsilon \rightarrow 0 }$ above
equals  $\sup_{\epsilon > 0}$.
\end{definition}

\begin{definition}\label{defn2}
The pre-image topological entropy for bundle RDS $T$ is defined as
$$
h_{\text{pre}}^{(r)}(T)=P_{\rm pre}(T,0)=\lim_{\epsilon \rightarrow
0 }\limsup_{n\rightarrow \infty}\frac{1}{n} \sup _{k\geq
n}\sup_{x\in X_\mathcal{E}}\int \log
s_n(\omega,\epsilon,(T^k_{\omega})^{-1}x)\,d\mathbf{P}(\omega),
$$
where $s_n(\omega,\epsilon,(T^k_{\omega})^{-1}x)$ is the largest
cardinality of an $(\omega, n,\epsilon )$-separated set of
$(T^k_{\omega})^{-1}x$.
\end{definition}

\begin{remark}
Definition \ref{defn2} is different from the random pre-image
topological entropy defined by Zhu \cite{Zhu}, which in our
terminology can be expressed as
$$
h_{\text{pre}}^{(r)}(T)=\lim_{\epsilon \rightarrow 0
}\limsup_{n\rightarrow \infty}\frac{1}{n} \sup _{k\geq n}\int
\log\sup_{x\in X_\mathcal{E}}
s_n(\omega,\epsilon,(T^k_{\omega})^{-1}x)\,d\mathbf{P}(\omega).
$$
The measurability of the function $\sup_{x\in X_\mathcal{E}}
s_n(\omega,\epsilon,(T^k_{\omega})^{-1}x)$ can not
 been guaranteed in most cases.
In Definition \ref{defn2}, based on Lemma \ref{lem},
$s_n(\omega,\epsilon,(T^k_{\omega})^{-1}x)$ is always measurable. On
the other hand, through a rigorous investigation, it is not hard to
see that if we replace Zhu' definition by Definition \ref{defn2} and
make a little change to Zhu's argument, then the variational
principle for the pre-image entropy for bundle RDS $T$ still holds,
namely, $h_{\text{pre}}^{(r)}(T)=\sup
\{h_{\text{pre},\,\mu}^{(r)}(T): \mu \in
\mathcal{M}_{\mathbf{P}}^1(\mathcal{E},T) \} $. In fact, we just
need to make some adjustment to the order of the supremum and the
logarithm and restrict the whole space $\Omega \times X$ to the
measurable subset $\mathcal{E}$ in his proof.
\end{remark}

\begin{remark}\label{rem1}
If the measure space $\Omega$ consists of a single point, i.e.,
$\Omega=\{\omega\}$, then bundle RDS $T$ reduces to a deterministic
dynamical system $(\mathcal{E}_{\omega}, d, T)$, where $T:
\mathcal{E}_{\omega} \rightarrow \mathcal{E}_{\omega}$ is
continuous. Furthermore, if $\mathcal{E}_{\omega}=X$, then
Definition \ref{defn1} is just the pre-image topological pressure
defined by Zeng \cite{Fanping} for deterministic dynamical systems
except for the order of the supremum and the logarithm, and the
difference between the two kinds of order  does not affect the
variational principle and the method of the argument.
\end{remark}

For a given $m\in \mathbb{N}_{+}$, if we replace $\vartheta$ by
$\vartheta^m$ and consider bundle RDS $T^m$ defined by
$(T^m)^n_{\omega}=T^m_{\vartheta^{(n-1)m}\omega}\cdots
T^m_{\vartheta^m \omega} T^m_{\omega}$, i.e.,
$(T^m)^n_{\omega}=T_{\omega}^{mn}$, then the pre-image topological
pressure has the following power rule.

\begin{proposition}\label{prop}
For any $m>0$, $P_{{\rm pre}}(T^m ,S_m f)=m P_{{\rm pre}}(T ,f)$.
\end{proposition}

\begin{proof}
Fix $m\in \mathbb{N}$. Let $n\in \mathbb{N}$, $k\geq n$ and $x\in X_
\mathcal{E}$. If $E$ is an $(\omega,n,\epsilon)$-separated set of
$((T^m)^k_{\omega})^{-1}x$ for $T^m$, then $E$ is also an
$(\omega,mn, \epsilon)$-separated set of $(T^{mk}_{\omega})^{-1}x$
for $T$. Since $(T^m)^n_{\omega}=T^{mn}_{\omega}$, so
$$
P_{\text{pre}, \,n, \,\omega}(T^m,
S_mf,\epsilon,((T^m)^k_{\omega})^{-1}x)\leq P_{\text{pre}, \,mn ,\,
\omega}(T,f, \epsilon, (T^{mk}_{\omega})^{-1}x)
$$
Hence
\begin{align*}
&P_{\text{pre}}(T^m, S_m f, \epsilon)\\
=&\limsup_{n\rightarrow \infty} \frac{1}{n}\sup_{k \geq n}
\sup_{x\in  X_\mathcal{E}} \int\log P_{\text{pre}, \,n,\,
\omega}(T^m, S_mf,\epsilon,((T^m)^k_{\omega})^{-1}x )\,
d\mathbf{P}(\omega)   \\
\leq & \limsup_{n\rightarrow \infty} \frac{1}{n}\sup_{k \geq n}
\sup_{x\in  X_\mathcal{E}} \int\log P_{\text{pre}, \,mn ,\,
\omega}(T,f, \epsilon, (T^{mk}_{\omega})^{-1}x) d\mathbf{P}(\omega)
\\
\leq & m \limsup_{n\rightarrow \infty} \frac{1}{mn}\sup_{k \geq mn}
\sup_{x\in  X_\mathcal{E}} \int\log P_{\text{pre}, \,mn ,\,
\omega}(T,f, \epsilon,
(T^k_{\omega})^{-1}x) d\mathbf{P}(\omega) \\
\leq & m\limsup_{n\rightarrow \infty} \frac{1}{n}\sup_{k \geq n}
\sup_{x\in  X_\mathcal{E}} \int\log P_{\text{pre}, \,n ,\,
\omega}(T,f, \epsilon,
(T^k_{\omega})^{-1}x) d\mathbf{P}(\omega)\\
 \leq &m P_{\text{pre}}(T, f, \epsilon).
\end{align*}
Therefore, $P_{\text{pre}}(T^m,S_m f)\leq m P_{\text{pre}}(T,f)$.

For any $\epsilon > 0$, by the continuity of $T^m$, there exists
some small enough $\delta >0$ such that if $d(y,z)\leq \delta$, $y,z
\in \mathcal{E}_{\omega}$, then $d^m_{\omega}(y,z)\leq \epsilon$.
For any positive integer $n$, there exists some integer $l$ such
that $mn\leq l < m(n+1)$. It is easy to see that any $(\omega, l,
\epsilon)$-separated set of $(T^k_{\omega})^{-1}x$ for $T$ is also
an $(\omega, n, \delta)$-separated set of $(T^k_{\omega})^{-1}x$ for
$T^m$. For $k \geq l$, let $k=k_1m+q$, where $k_1,q \in \mathbb{N},
0\leq q <m$, then $k_1\geq n$. Let $x'=T_{\vartheta^{k+m-q-1}\omega}
\cdots T_{\vartheta^k\omega}x $, then $x\in
(T_{\vartheta^{k+m-q-1}\omega} \cdots
T_{\vartheta^k\omega})^{-1}x'$, and then $(T_{\omega}^k)^{-1}x =
(T_{\omega}^{k_1m+q})^{-1}x\subset
(T_{\omega}^{(k_1+1)m})^{-1}x'=((T^m)_{\omega}^{k'})^{-1}x'$, where
$k'=k_1+1$. Hence, any $(\omega,n,\delta)$-separated set of
$(T_{\omega}^k)^{-1}x$ for $T^m$ is also an
$(\omega,n,\delta)$-separated set of $((T^m)_{\omega}^{k'})^{-1}x'$
for $T^m$. Therefore, for $k \geq l$,
\begin{align*}
& P_{\text{pre},\,l,\, \omega}(T,f,\epsilon,
(T^k_{\omega})^{-1}x)\\
&= \sup \bigl\{ \sum_{ y\in E}\exp S_lf(\omega,y): E \, \text{is an}
\, (\omega,l, \epsilon)\, \text{-separated set of}\,\,
(T^k_{\omega})^{-1}x \,\, \text{for}\,\,T
\bigr\} \\
\begin{split}
= \sup \bigl\{ \sum_{y\in E}\exp
(\sum_{i=0}^{n-1}S_mf(\vartheta^{mi}\omega, T^{mi}_{\omega}y)+
\sum^{l-1}_{j=mn}f(\vartheta^j\omega ,T_{\omega}^jy)):\quad\quad\quad\quad\quad\quad \\
E \, \text{is an} \, (\omega,l, \epsilon)\, \text{-separated set
of}\,\, (T^k_{\omega})^{-1}x \,\, \text{for}\,\,T \bigr \}
\end{split}\\
\begin{split}
\leq \exp \sum_{j=mn}^{l-1}\|f(\vartheta^j\omega)\|_{\infty} \sup
\bigl\{ \sum_{y\in E}\exp
\sum_{i=0}^{n-1}S_mf((\vartheta^m)^i\omega,
(T^m)^i_{\omega}y):\quad\quad\quad\quad\quad\quad  \\
E \, \text{is an} \, (\omega,n, \delta)\, \text{-separated set
of}\,\, (T^k_{\omega})^{-1}x \,\, \text{for}\,\, T^m \bigr\}
\end{split} \\
\begin{split}
\leq \exp \sum_{j=mn}^{l-1}\|f(\vartheta^j\omega)\|_{\infty} \sup
\bigl\{ \sum_{y\in E}\exp
\sum_{i=0}^{n-1}S_mf((\vartheta^m)^i\omega,
(T^m)^i_{\omega}y):\quad\quad\quad\quad\quad\quad  \\
E \, \text{is an} \, (\omega,n, \delta)\, \text{-separated set
of}\,\, ((T^m)_{\omega}^{k'})^{-1}x' \,\, \text{for}\,\, T^m \bigr\}
\end{split}\\
\begin{split}
= \exp \sum_{j=mn}^{l-1}\|f(\vartheta^j\omega)\|_{\infty} P_{{\rm
pre},\,n,\,\omega}(T^m, S_m f, \delta,
((T^m)_{\omega}^{k'})^{-1}x').
\end{split}
\end{align*}
Since $f \in \mathbf{L}_{\mathcal{E}}^1 (\Omega, \mathcal{C}(X))$,
so $\int \sum_{j=mn}^{l-1}\|f(\vartheta^j\omega)\|_{\infty}
\,d\mathbf{P}(\omega) \leq m\|f\|<\infty$. Then by the definition of
$P_{\text{pre}}(T,f, \epsilon)$, we have
\begin{align*}
&P_{\text{pre}}(T, f, \epsilon)\\
&=\limsup_{l\rightarrow \infty} \frac{1}{l}\sup_{k \geq l}
\sup_{x\in X_\mathcal{E}} \int\log P_{\text{pre}, \,l,\, \omega}(T,
f,\epsilon,(T^k_{\omega})^{-1}x )\,
d\mathbf{P}(\omega)\\
\begin{split}
\leq \limsup_{n\rightarrow \infty} \frac{1}{mn}\sup_{k' \geq n}
\sup_{x'\in X_\mathcal{E}} \int \bigl(
\sum_{j=mn}^{l-1}\|f(\vartheta^j\omega)\|_{\infty}
\quad\quad\quad\quad\quad\quad\quad\quad\quad\quad\quad\quad\quad\quad
\\
+ \log P_{\text{pre}, \,n,\, \omega}(T^m,
S_mf,\delta,((T^m)^{k'}_{\omega})^{-1}x' ) \bigr)\,
d\mathbf{P}(\omega)
\end{split}\\
&=\frac{1}{m}\limsup_{n\rightarrow \infty} \frac{1}{n}\sup_{k' \geq
n} \sup_{x'\in X_\mathcal{E}} \int\log P_{\text{pre}, \,n,\,
\omega}(T^m, S_mf,\delta,((T^m)^{k'}_{\omega})^{-1}x' ) \,
d\mathbf{P}(\omega)
\\
&=\frac{1}{m}P_{\rm pre}(T^m,S_mf,\delta).
\end{align*}
Hence,
$$
mP_{\text{pre}}(T,f, \epsilon)\leq P_{\text{pre}}(T^m,S_mf,\delta).
$$
If $\epsilon\rightarrow 0$, then $\delta\rightarrow 0$. Hence we
obtain $mP_{\text{pre}}(T,f)\leq P_{\text{pre}}(T^m,S_mf)$ and
complete the proof.
\end{proof}

\section{Pre-image variational principle for bundle RDS}\label{Sec4}

\begin{theorem}\label{theorem}
If $T$ is a continuous bundle RDS on $\mathcal{E}$ and $f\in
\mathbf{L}^1_{\mathcal{E}}(\Omega, \mathcal{C}(X))$, then
$$
P_{{\rm pre}}(T,f)=\sup\{ h_{{\rm pre},\,\mu}^{(r)}(T)+ \int f\,d\mu
:\mu \in \mathcal{M}_{\mathbf{P}}^1(\mathcal{E},T)  \}.
$$
\end{theorem}

\begin{proof}
 (1)
Let $\mu \in \mathcal{M}_{\mathbf{P}}^1(\mathcal{E},T) $,
  $\mathcal{P}=\{P_1, \cdots,
 P_k\}$  be a finite measurable partition of $X$, and $\epsilon$ be
 a  positive number with $\epsilon k \log k <1$.
 Let $\mathcal{P}(\omega)=\{P_1(\omega), \cdots,
 P_k(\omega)\}$ be the corresponding partition of $\mathcal{E}_{\omega}$,
 where $P_i(\omega)=P_i \cap \mathcal{E}_{\omega}, i=1,\cdots,
 k$.  By the
 regularity of $\mu$, we can find compact sets $Q_i \subset P_i, 1 \leq i \leq
 k$, such that
 \begin{equation*}
 \int \mu_{\omega}(P_i(\omega)\backslash
Q_i(\omega)) \, d\mathbf{P}(\omega) < \epsilon,
 \end{equation*}
where $Q_i(\omega)= Q_i \cap \mathcal{E}_{\omega}$. Let
$\mathcal{Q}(\omega)=\{Q_0(\omega),\cdots, Q_k(\omega) \}$ be the
partition of $\mathcal{E}_{\omega}$, where $Q_0(\omega)=
\mathcal{E}_{\omega} \backslash \bigcup_{i=1}^kQ_i(\omega)$. Then by
the proof  of \cite[Theorem 4]{Zhu} and   \cite[Lemma II.
1.3]{Kifer}, the following inequality holds (See \cite{Zhu} for
details),
\begin{equation}\label{entr}
h_{\text{pre},\,\mu}^{(r)}(T, \Omega \times \mathcal{P})\leq
h_{\text{pre},\,\mu}^{(r)}(T, \Omega \times \mathcal{Q}) +1,
\end{equation}
where $\Omega \times \mathcal{P}$ (respectively by $\Omega \times
\mathcal{Q}$) is the partition of $ \Omega \times X$ into sets
$\Omega \times P_i$ (respectively by $\Omega \times Q_i$).

Let
$$
\mathcal{Q}_n(\omega)=
\bigvee_{i=0}^{n-1}(T^i_{\omega})^{-1}\mathcal{Q}(\vartheta^i\omega).
$$
For each $k$, we choose a non-decreasing sequence of finite
partitions $\beta^k_1\leq \beta^k_2\leq \cdots$ with diameters
tending to zero for which
$\mathcal{B}_{\vartheta^k\omega}=\bigvee_{j=1}^{\infty}\beta_j^k$ up
to set of measure 0, and satisfying
\begin{equation*}
H_{\mu_{\omega}}(\mathcal{Q}_n(\omega)|
(T^k_{\omega})^{-1}\mathcal{B}_{\vartheta^k\omega})=\lim_{j\rightarrow
\infty
}H_{\mu_{\omega}}(\mathcal{Q}_n(\omega)|(T^k_{\omega})^{-1}\beta_j^k
).
\end{equation*}

Let $\eta=\{Q_0\cup Q_1, \cdots, Q_0\cup Q_k\}$, and $\tau$ be the
Lebesgue number of $\eta$. We pick a small nonrandom $\delta$ with
$0<4\delta<\tau$. Let $\epsilon_1:=\epsilon_1(\omega,n, \epsilon)>0$
for each $\omega\in \Omega$, such that if $d(x,y)<\epsilon_1$, $x, y
\in \mathcal{E}_{\omega}$, then $d^{\omega}_n(x,y)<\delta$.

For each $\omega\in \Omega$, $\{(T^k_{\omega})^{-1}x: x \in
\mathcal{E}_{\vartheta^k\omega}\}$ is an upper semi-continuous
decomposition of $\mathcal{E}_{\omega}$. Thus for each $x\in
\mathcal{E}_{\vartheta^k\omega}$, there is an
$\epsilon_2:=\epsilon_2(\omega, x, k, \epsilon_1)$ such that if
$d(x,y)<\epsilon_2$, $y \in \mathcal{E}_{\vartheta^k\omega}$ and
$y_1\in (T^k_{\omega})^{-1}y$, then there is an $x_1 \in
(T^k_{\omega})^{-1}x$ such that $d(x_1,y_1)<\epsilon_1$. Let
$\mathcal{U}_k$ be the collection of open $\epsilon_2$ balls in
$\mathcal{E}_{\vartheta^k\omega}$ as $x$ varies in
$\mathcal{E}_{\vartheta^k\omega}$ and  $\epsilon_3$ be the Lebesgue
number for $\mathcal{U}_k$.

Since diam$(\beta_j^k)\rightarrow 0$ as $j\rightarrow \infty$, we
can choose large enough $j$ such that for each $B\in \beta_j^k$,
diam$\overline{B}<\epsilon_3$. For each $A \in
(T^k_{\omega})^{-1}\beta_j^k$, let $\mu_{\omega,A}$ be the
conditional measure of $\mu_{\omega}$ restricted to $A$,
$\mathcal{Q}_n(\omega,A)=\{A\cap C: C \in \mathcal{Q}_n(\omega),
A\cap C \neq \emptyset \}$, $ S_n^*f(\omega,A\cap
C)=\sup\{S_nf(\omega,x): x\in A\cap C \} $. Then by the standard
inequality
$$ \sum_{i=1}^m p_i(a_i-\log p_i)\leq \log
\sum_{i=1}^m\exp a_i ,$$
 for any probability vector $(p_1,\cdots,
p_m)$, we have
\begin{equation}\label{inseq2}
\begin{split}
&H_{\mu_{\omega}}(\mathcal{Q}_n(\omega)|(T^k_{\omega})^{-1}\beta_j^k)+
\int_{\mathcal{E}_{\omega}}S_nf\,d\mu_{\omega}\\
&=\sum_{A\in (T^k_{\omega})^{-1}\beta_j^k} \bigl(
\mu_{\omega}(A)H_{\mu_{\omega,A}}(\mathcal{Q}_n(\omega, A))+\int_A
S_nf \,d\mu_{\omega}
\bigr)\\
&\leq \sum_{A\in
(T^k_{\omega})^{-1}\beta_j^k}\mu_{\omega}(A)\!\!\!\!\sum_{A\cap C\in
\mathcal{Q}_n(\omega, A)}\!\!\!\! \mu_{\omega, A}(A\cap C)
\bigl(-\log
\mu_{\omega, A}(A\cap C) + S_n^*f (\omega, A\cap C) \bigr) \\
&\leq \max_{A\in (T^k_{\omega})^{-1}\beta_j^k} \log
\!\!\!\!\sum_{A\cap C\in \mathcal{Q}_n(\omega, A)} \!\!\!\!\exp
S_n^*f (\omega, A\cap C).
\end{split}
\end{equation}

In the sequel, let $A$ be the maximal one satisfying the above
inequality and $B\in \beta_j^k$ with $A=(T^k_{\omega})^{-1}B$. For
each $A\cap C \in \mathcal{Q}_n(\omega, A)$, we choose some point
$x_C \in \overline{A\cap C}$ such that $S_nf(\omega, x_C)=S_n^*f
(\omega, A\cap C)$.

Since $T^k_{\omega} (x_C) \in \overline{B}$, and
diam$\overline{B}<\epsilon_3$, there is an $u_B\in
\mathcal{E}_{\vartheta^k\omega} $ such that $d(u_B,T^k_{\omega}
(x_C))$ $<\epsilon_2$. Hence there is $y_1 \in (T^k_{\omega})^{-1}
u_B$ such that $d(x_C, y_1)<\epsilon_1$ and then $d_n^{\omega}(x_C,
y_1)<\delta$.

Let $E_A^{\omega}$ be a maximal $(\omega, n, \delta)$-separated set
in $(T^k_{\omega})^{-1} u_B$. Since $E_A^{\omega}$ is also a
spanning set, there is a point $z(x_C)\in E_A^{\omega}$ such that
$d_n^{\omega}(y_1,z(x_C))\leq \delta$, then
$d_n^{\omega}(x_C,z(x_C))\leq 2\delta$ and hence
\begin{equation}\label{inseq1}
S_n^*f(\omega, A\cap C)\leq S_nf (\omega, z(x_C))+ \sum_{i=0}^{n-1}
\kappa_{2\delta}^{(f)}(\vartheta^i\omega).
\end{equation}
We now show that if $y\in E_A^{\omega}$, then
\begin{equation}\label{card}
\text{card}\{A\cap C \in \mathcal{Q}_n(\omega, A): z(x_C)=y\}\leq
2^n.
\end{equation}
 Suppose $z(x_C)=z(x_{\widetilde{C}})$. There exist $ x_C\in \overline{A\cap
 C}$ and $x_{\widetilde{C}}\in \overline{A\cap \widetilde{C}}$ so
 that $d_n^{\omega}(x_C,z(x_C) )< 2\delta$ and
 $d_n^{\omega}(x_{\widetilde{C}},z(x_{\widetilde{C}}))<2\delta$.
 Therefore $d_n^{\omega}(x_C,x_{\widetilde{C}})<4\delta$; so
 $T_{\omega}^i(x_C)$ and $T_{\omega}^i(x_{\widetilde{C}})$ are in
 the same element of $\eta$, say $Q_0\cup Q_{j_i}, 0\leq i<n$. Hence
 there are at most $2^n$ elements $A\cap C$ of $\mathcal{Q}_n(\omega,
 A)$ equal to a fixed member of $E_A^{\omega}$.

Combining \eqref{inseq1} and \eqref{card}, we get
\begin{equation*}
\sum_{A\cap C\in \mathcal{Q}_n(\omega, A)}\exp S_n^*f (\omega, A\cap
C) \leq 2^n  \sum_{y\in E_A^{\omega}} \exp
\bigl(S_nf(\omega,y)+\sum_{i=0}^{n-1}
\kappa_{2\delta}^{(f)}(\vartheta^i\omega) \bigr).
\end{equation*}
Taking the logarithm of both parts and using the resulting
inequality in order to estimate the righthand side of
\eqref{inseq2}, we obtain
\begin{align*}
&H_{\mu_{\omega}}(\mathcal{Q}_n(\omega)|(T^k_{\omega})^{-1}\beta_j^k)+
\int_{\mathcal{E}_{\omega}}S_nf\,d\mu_{\omega}\\
 \leq &n\log 2 +\log \sum_{y\in E_A^{\omega}} \exp S_nf(\omega,y)+\sum_{i=0}^{n-1}
\kappa_{2\delta}^{(f)}(\vartheta^i\omega)\\
\leq &n\log 2 + \log P_{\text{pre},n, \omega}(T, f, \delta,
(T_{\omega}^k)^{-1}u_B)+\sum_{i=0}^{n-1}
\kappa_{2\delta}^{(f)}(\vartheta^i\omega).
\end{align*}
Integrating this with respect to $\mathbf{P}$ and letting
$j\rightarrow \infty$, we have
\begin{align*}
&\int H_{\mu_{\omega}}(\mathcal{Q}_n(\omega)|
(T^k_{\omega})^{-1}\mathcal{B}_{\vartheta^k\omega})d\mathbf{P}(\omega)+
\int S_nf\,d\mu\\
\leq &n\log 2 + \int \log P_{\text{pre},n, \omega}(T, f, \delta,
(T_{\omega}^k)^{-1}u_B)d\mathbf{P}(\omega) + \int\sum_{i=0}^{n-1}
\kappa_{2\delta}^{(f)}(\vartheta^i\omega) d\mathbf{P}(\omega)\\
\leq &n\log 2 + \sup_{x\in \mathcal{E}_{\vartheta^k\omega}} \int
\log P_{\text{pre},n, \omega}(T, f, \delta,
(T_{\omega}^k)^{-1}x)d\mathbf{P}(\omega) + \int\sum_{i=0}^{n-1}
\kappa_{2\delta}^{(f)}(\vartheta^i\omega) d\mathbf{P}(\omega).
\end{align*}
Letting $k\rightarrow \infty$, dividing by $n$ and letting
$n\rightarrow \infty$, by Proposition \ref{eq1} and the equality
\eqref{predefn}, we have
\begin{align*}
h^{(r)}_{\text{pre},\,\mu}(T, \Omega\times \mathcal{Q})+ \int
f\,d\mu \leq \log 2 + P_{\text{pre}}(T,f,\delta) + \int
\kappa_{2\delta}^{(f)} d\mathbf{P}(\omega).
\end{align*}
Using \eqref{entr}, we derive the inequality
$$
h^{(r)}_{\text{pre},\,\mu}(T, \Omega\times \mathcal{P})+ \int
f\,d\mu \leq \log 2 +1 + P_{\text{pre}}(T,f,\delta) + \int
\kappa_{2\delta}^{(f)} d\mathbf{P}(\omega).
$$
Since this is true for all finite partitions $\mathcal{P}$ of $X$
and all positive $\delta$, we have
$$
h^{(r)}_{\text{pre},\,\mu}(T)+\int f\,d\mu \leq \log 2 +1 +
P_{\text{pre}}(T,f).
$$
Since $h^{(r)}_{\text{pre},\,\mu}(T^m)=m
h^{(r)}_{\text{pre},\mu}(T)$ (See \cite[Proposition 4]{Zhu} ), the
same arguments as above applied to $T^n$ and to $S_nf$  yield
$$
n(h^{(r)}_{\text{pre},\,\mu}(T )+\int f\,d\mu )\leq \log 2 +1 +
P_{\text{pre}}(T^n,S_nf).
$$
Taking into account Proposition \ref{prop}, dividing by $n$ and
letting $n\rightarrow \infty$, we conclude that
$$
h^{(r)}_{\text{pre},\,\mu}(T )+\int f\,d\mu \leq
P_{\text{pre}}(T,f).
$$

(2) By the equality \ref{predefn}, we can choose a sequence
$n_i\rightarrow \infty$, $k_i\geq n_i$, and points
$x_{\vartheta^{k_i}\omega}\in \mathcal{E}_{\vartheta^{k_i}\omega}$
for each $\omega \in \Omega$ such that
\begin{equation*}
P_{\text{pre}}(T,f, \epsilon)=\lim_{i\rightarrow \infty}
\frac{1}{n_i}\int \log P_{\text{pre},\, n_i, \,\omega}(T,f,
\epsilon,
(T_{\omega}^{k_i})^{-1}x_{\vartheta^{k_i}\omega})\,d\mathbf{P}(\omega).
\end{equation*}
For a small nonrandom $\epsilon>0$, by Lemma \ref{lem}, we can
choose a family of maximal $(\omega, n_i,\epsilon )$-separated sets
$G(\omega, n_i,\epsilon)\subset
(T_{\omega}^{k_i})^{-1}x_{\vartheta^{k_i}\omega}$ $ \subset
\mathcal{E}_{\omega}$, which are measurable in $\omega$, such that
\begin{equation}\label{inseq7}
\sum_{x\in G(\omega, n_i,\epsilon)}\exp S_{n_i}f(\omega,x)\geq
\frac{1}{e}P_{\text{pre},\, n_i, \,\omega}(T,f, \epsilon,
(T_{\omega}^{k_i})^{-1}x_{\vartheta^{k_i}\omega}).
\end{equation}
Next, we define probability measures $\nu^{(i)}$ on $\mathcal{E}$
via their measurable disintegrations
\begin{equation*}
\nu^{(i)}_{\omega}=\frac{\sum_{x\in G(\omega, n_i,\epsilon)}\exp
S_{n_i}f(\omega,x)\delta_x }{\sum_{x\in G(\omega, n_i,\epsilon)}\exp
S_{n_i}f(\omega,x)}
\end{equation*}
so that
$d\nu^{(i)}(\omega,x)=d\nu^{(i)}_{\omega}(x)\,d\mathbf{P}(\omega)$.

Let
\begin{equation*}
\mu^{(i)}=\frac{1}{n_i}\sum_{j=0}^{n_i-1}\Theta^j\nu^{(i)}.
\end{equation*}
Then, by Lemma 2.1 (i)--(ii) of \cite{Kifer2001}, we can choose a
subsequence $n_{i_l}$ of $\{n_i\}$ such that $\lim_{l\rightarrow
\infty}\mu^{(i_l)}=\mu$ for some $\mu \in \mathcal
{M}_{\mathbf{P}}^1(\mathcal{E},T)$. Without loss of generality, we
still assume that $\lim_{i\rightarrow \infty}\mu^{(i)}=\mu $.

Next, we choose a partition $\mathcal{P}=\{P_1,\cdots,P_k\}$ of $X$
with $\text{diam}\mathcal{P}\leq \epsilon$, which satisfies $\int
\mu_{\omega}(\partial P_i) \, d\mathbf{P}(\omega)=0$ for all $1\leq
i \leq k$, where $\partial$ denotes the boundary. Let
$\mathcal{P}(\omega)=\{P_1(\omega), \cdots, P_k(\omega)\}$,
$P_i(\omega)=P_i \cap \mathcal{E}_{\omega}$. Since each element of
$\bigvee_{l=0}^{n_i-1}(T^l_{\omega})^{-1}\mathcal{P}(\vartheta^l\omega)$
contains at most one element of $G(\omega, n_i,\epsilon)$, then by
the inequality \eqref{inseq7}, we have
\begin{equation}\label{inseq3}
\begin{split}
&H_{\nu^{(i)}_{\omega}}\bigl(
\bigvee_{l=0}^{n_i-1}(T^l_{\omega})^{-1}\mathcal{P}(\vartheta^l\omega)
|_{(T^{k_i}_{\omega})^{-1}x_{\vartheta^{k_i}\omega}} \bigr)+ \int
S_{n_i}f(\omega)\,d\nu^{(i)}_{\omega}\\
=&\sum_{x\in G(\omega,
n_i,\epsilon)}\nu^{(i)}_{\omega}(\{x\})(-\log\nu^{(i)}_{\omega}(\{x\})+
S_{n_i}f(\omega,x)
)\\
=&\log \sum_{x\in G(\omega, n_i,\epsilon)}\exp S_{n_i}f(\omega,x)\\
\geq &\log P_{\text{pre},\, n_i, \,\omega}(T,f, \epsilon,
(T_{\omega}^{k_i})^{-1}x_{\vartheta^{k_i}\omega})-1.
\end{split}
\end{equation}

For $\omega\in \Omega$, let $\mathcal{C}_{\omega}$ be the
subcollection of $\mathcal{B}^-_{\omega}$ consisting of
$\mu_{\omega}$-null sets. For any $\sigma$-algebra $\mathcal{A}$ of
subsets of $\mathcal{E}_{\omega}$, there is an enlarged
$\sigma$-algebra $\mathcal{A}_{\mathcal{C}_{\omega}}$ defined by
$A\in \mathcal{A}_{\mathcal{C}_{\omega}}$ if and only if there are
sets $B, M, N$ such that $A=B\cup M, B\in \mathcal{A}, N\in
\mathcal{C}_{\omega}$ and $M\subset N$. The $\sigma$-algebra
$\mathcal{B}^-_{\mathcal{C}_{\omega}}$ is simply the standard
$\mu_{\omega}$-completion of $\mathcal{B}^-_{\omega}$. Let
$\mathcal{B}^k_{\omega}=((T^k_{\omega})^{-1}
\mathcal{B}_{\vartheta^k\omega} )_{\mathcal{C}_{\omega}}$ for all
$k\geq 1$. Since
$T^{-1}_{\omega}\mathcal{C}_{\vartheta\omega}\subset
\mathcal{C}_{\omega}$ for each $\omega$, we have that
$\mathcal{B}_{\omega}^1 \supset \mathcal{B}_{\omega}^2 \supset
\cdots$. Let $\mathcal{B}_{\omega}^{\infty}= \bigcap_{k\geq
1}\mathcal{B}_{\omega}^k$, then $\mathcal{B}_{\omega}^-\subset
\mathcal{B}_{\mathcal{C}_{\omega}}^-\subset
\mathcal{B}_{\omega}^{\infty}$ and
$(T_{\omega}^l)^{-1}\mathcal{B}^k_{\vartheta^l\omega}\subset
\mathcal{B}_{\omega}^{l+k}$ for all $l\geq 1$.

Similarly, let $\mathcal{C}_{\mathcal{E}}$ be the subcollection of
$(\mathcal{F}\times\mathcal{B})^-_{\mathcal{E}} $ consisting of
$\mu$-null sets, $(\mathcal{F}\times\mathcal{B})^k_{\mathcal{E}}=
(\Theta^{-k}(\mathcal{F}\times\mathcal{B})_{\mathcal{E}})
_{\mathcal{C}_{\mathcal{E}}}$ and
$(\mathcal{F}\times\mathcal{B})^{\infty}_{\mathcal{E}}=\bigcap_{k\geq
1}(\mathcal{F}\times\mathcal{B})^k_{\mathcal{E}}$. Clearly,
$(\mathcal{F}\times\mathcal{B})^1_{\mathcal{E}}\supset
(\mathcal{F}\times\mathcal{B})^2_{\mathcal{E}} \supset \cdots$,
$(\mathcal{F}\times\mathcal{B})^-_{\mathcal{E}}\subset
(\mathcal{F}\times\mathcal{B})^{\infty}_{\mathcal{E}} \subset
(\mathcal{F}\times\mathcal{B})^k_{\mathcal{E}}$ and
$\Theta^{-l}(\mathcal{F}_{\mathcal{E}}\vee
(\mathcal{F}\times\mathcal{B})^k_{\mathcal{E}})\subset
\mathcal{F}_{\mathcal{E}} \vee
(\mathcal{F}\times\mathcal{B})^{l+k}_{\mathcal{E}}$ for all $l\geq
1$. Similar to the proof of Proposition \ref{eq1}, we conclude that
for each $k$ and any finite partition $O=\{O_i\}$ of $\mathcal{E}$,
\begin{equation}\label{seq1}
H_{\nu^{(i)}}(O|\mathcal{F}_{\mathcal{E}}\vee
(\mathcal{F}\times\mathcal{B})^k_{\mathcal{E}})= \int
H_{\nu^{(i)}_{\omega}}(O(\omega)|\mathcal{B}^k_{\omega})\,
d\mathbf{P}(\omega),
\end{equation}
where $O(\omega)=\{O_i(\omega)\}$, $O_i(\omega)=\{x\in
\mathcal{E}_{\omega}: (\omega,x)\in O_i\}$.

Note that $\nu^{(i)}_{\omega}$ is supported on
$(T_{\omega}^{k_i})^{-1}x_{\vartheta^{k_i}\omega}$, the canonical
system of conditional measures induced by $\nu^{(i)}_{\omega}$ on
the measurable partition $\{(T_{\omega}^{k_i})^{-1}x|x\in
\mathcal{E}_{\vartheta^{k_i}\omega} \}$ reduces to a single measure
on the set $(T_{\omega}^{k_i})^{-1}x_{\vartheta^{k_i}\omega}$, which
 may be identified with $\nu^{(i)}_{\omega}$. Now, each element
$A\in \mathcal{B}_{\omega}^{k_i}$ can be expressed as the disjoint
union $A= B\cup C$ with $B\in
(T_{\omega}^{k_i})^{-1}\mathcal{B}_{\vartheta^{k_i}\omega}$ and
$C\in \mathcal{C}_{\omega}$.  Since $\nu^{(i)}_{\omega}$ is
supported on elements of
$(T_{\omega}^{k_i})^{-1}x_{\vartheta^{k_i}\omega}$, we have
$\nu^{(i)}_{\omega}(C)=0$. Hence for any finite partition $\gamma$
of $\mathcal{E}_{\omega}$, we have
\begin{equation}\label{inseq4}
H_{\nu^{(i)}_{\omega}}(\gamma|\mathcal{B}^{k_i}_{\omega})
=H_{\nu^{(i)}_{\omega}}(\gamma|
_{(T_{\omega}^{k_i})^{-1}x_{\vartheta^{k_i}\omega}}).
\end{equation}

Let $\mathcal{Q}=\{Q_1,\cdots, Q_k\}$, where $Q_i=(\Omega\times
P_i)\cap\mathcal{E}$, then $\mathcal{Q}$ is a partition of
$\mathcal{E}$ and $Q_i(\omega)=\{x\in \mathcal{E}_{\omega}:
(\omega,x)\in Q_i \}=P_i(\omega)$. Integrating in \eqref{inseq3}
with respect to $\mathbf{P}$, then by \eqref{seq1}, \eqref{inseq4}
and $\int S_{n_i}f \,d\nu^{(i)}=n_i \int f \,d\mu^{(i)}$, we obtain
the inequality
\begin{equation}\label{inseq5}
\begin{split}
H_{\nu^{(i)}}\bigl(\bigvee_{l=0}^{n_i-1}(\Theta^l)^{-1}\mathcal{Q}|\mathcal{F}_{\mathcal{E}}\vee
(\mathcal{F}\times\mathcal{B})^{k_i}_{\mathcal{E}} \bigr) + n_i \int
f \,
d\mu^{(i)}\\
 \geq \int \log P_{\text{pre},\, n_i, \,\omega}(T,f,
\epsilon,(T_{\omega}^{k_i})^{-1}x_{\vartheta^{k_i}\omega})\,d\mathbf{P}(\omega)-1.
\end{split}
\end{equation}
For $q, n_i\in \mathbb{N}$ with $1<q< n_i$, let $a(s)$ denote the
integer part of $(n_i-s)q^{-1}$ for all $0\leq s<q$. Then, clearly,
for each $s$, we have
\begin{equation*}
\bigvee_{l=0}^{n_i-1}(\Theta^l)^{-1}\mathcal{Q}
=\bigvee_{r=0}^{a(s)-1}(\Theta^{rq+s})^{-1}\bigvee_{t=0}^{q-1}(\Theta^t)^{-1}\mathcal{Q}\vee
\bigvee_{l\in S}(\Theta^l)^{-1}\mathcal{Q} ,
\end{equation*}
where $\text{card}S\leq 2q$.

Since $\text{card}\mathcal{Q}=k$, by the subadditivity of
conditional entropy (See \cite[Section 2.1]{Kifer}), we have
\begin{align*}
&H_{\nu^{(i)}}\bigl(\bigvee_{l=0}^{n_i-1}(\Theta^l)^{-1}\mathcal{Q}|\mathcal{F}_{\mathcal{E}}\vee
(\mathcal{F}\times\mathcal{B})^{k_i}_{\mathcal{E}} \bigr)\\
\leq &\sum_{r=0}^{a(s)-1}H_{\nu^{(i)}}
\bigl((\Theta^{rq+s})^{-1}\bigvee_{t=0}^{q-1}(\Theta^t)^{-1}\mathcal{Q}|\mathcal{F}_{\mathcal{E}}\vee
(\mathcal{F}\times\mathcal{B})^{k_i}_{\mathcal{E}}
 \bigr)
+2q\log k\\
\leq &\sum_{r=0}^{a(s)-1}H_{\nu^{(i)}}
\bigl((\Theta^{rq+s})^{-1}\bigvee_{t=0}^{q-1}(\Theta^t)^{-1}\mathcal{Q}|
(\Theta^{rq+s})^{-1}(\mathcal{F}_{\mathcal{E}}\vee
(\mathcal{F}\times\mathcal{B})^{k_i}_{\mathcal{E}})\bigr)
+2q\log k\\
=& \sum_{r=0}^{a(s)-1}H_{\Theta^{rq+s}\nu^{(i)}}
\bigl(\bigvee_{t=0}^{q-1}(\Theta^t)^{-1}\mathcal{Q}|\mathcal{F}_{\mathcal{E}}\vee
(\mathcal{F}\times\mathcal{B})^{k_i}_{\mathcal{E}}\bigr) +2q\log k .
\end{align*}
Summing this inequality over $s\in \{0,1,\cdots,q-1\}$, we get
\begin{align*}
&qH_{\nu^{(i)}} \bigl(
\bigvee_{l=0}^{n_i-1}(\Theta^l)^{-1}\mathcal{Q}|\mathcal{F}_{\mathcal{E}}\vee
(\mathcal{F}\times\mathcal{B})^{k_i}_{\mathcal{E}}
\bigr)\\
\leq &\sum_{l=0}^{n_i-1}H_{\Theta^l\nu^{(i)}}| \bigl(
\bigvee_{t=0}^{q-1}(\Theta^t)^{-1}\mathcal{Q}|\mathcal{F}_{\mathcal{E}}\vee
(\mathcal{F}\times\mathcal{B})^{k_i}_{\mathcal{E}} \bigr)
+2q^2\log k\\
\leq &n_iH_{\mu^{(i)}} \bigl(
\bigvee_{t=0}^{q-1}(\Theta^t)^{-1}\mathcal{Q}|\mathcal{F}_{\mathcal{E}}\vee
(\mathcal{F}\times\mathcal{B})^{k_i}_{\mathcal{E}} \bigr)
+2q^2\log k\\
\leq &n_iH_{\mu^{(i)}} \bigl(
\bigvee_{t=0}^{q-1}(\Theta^t)^{-1}\mathcal{Q}|\mathcal{F}_{\mathcal{E}}\vee
(\mathcal{F}\times\mathcal{B})^{\infty}_{\mathcal{E}} \bigr)
+2q^2\log k ,
\end{align*}
where the second inequality, as in Kifer's works \cite{Kifer2001},
relies on the general property of conditional entropy of partition
$\sum_ip_iH_{\eta_i}(\xi|\mathcal{A})\leq
H_{\sum_ip_i\eta_i}(\xi|\mathcal{A})$ which holds for any finite
partition $\xi$, $\sigma$-algebra $\mathcal{A}$, probability
measures $\eta_i$, and probability vector $(p_i), i=1,\ldots, n$, in
view of the convexity of $t\log t$ in the same way as that in the
unconditional case (cf. \cite[pp.183 and 188 ]{Walters} ). This
together with \eqref{inseq5} yields
\begin{align*}
&q\int \log P_{\text{pre},\, n_i, \,\omega}(T,f,
\epsilon,(T_{\omega}^{k_i})^{-1}x_{\vartheta^{k_i}\omega})\,d\mathbf{P}(\omega)-q\\
\leq &n_iH_{\mu^{(i)}} \bigl(
\bigvee_{t=0}^{q-1}(\Theta^t)^{-1}\mathcal{Q}|\mathcal{F}_{\mathcal{E}}\vee
(\mathcal{F}\times\mathcal{B})^{\infty}_{\mathcal{E}}
 \bigr) +2q^2\log k
+ q n_i \int f\,d\mu^{(i)}.
\end{align*}
Diving by $n_i$, passing to the $\limsup_{i\rightarrow \infty}$ and
using the inequality 10 in  \cite{Zhu}, i.e.,
\begin{align*}
&H_{\mu} \bigl (
\bigvee_{t=0}^{q-1}(\Theta^t)^{-1}\mathcal{Q}|\mathcal{F}_{\mathcal{E}}\vee
(\mathcal{F}\times\mathcal{B})^-_{\mathcal{E}}
 \bigr)\\
  \geq
  &\limsup_{n\rightarrow \infty}H_{\mu^{(i)}} \bigl(
\bigvee_{t=0}^{q-1}(\Theta^t)^{-1}\mathcal{Q}|\mathcal{F}_{\mathcal{E}}\vee
(\mathcal{F}\times\mathcal{B})^{\infty}_{\mathcal{E}} \bigr),
\end{align*}
we get
$$
qP_{\text{pre}}(T,f,\epsilon) \leq H_{\mu} \bigl(
\bigvee_{t=0}^{q-1}(\Theta^t)^{-1}\mathcal{Q}|\mathcal{F}_{\mathcal{E}}\vee
(\mathcal{F}\times\mathcal{B})^-_{\mathcal{E}}
 \bigr)+q\int f\, d\mu.
$$
Dividing by $q$ and letting $q\rightarrow \infty$, we have
\begin{align*}
P_{\text{pre}}(T,f,\epsilon) \leq
h_{\text{pre},\,\mu}^{(r)}(T,\mathcal{Q}) + \int f\,d\mu \leq
h_{\text{pre},\,\mu}^{(r)}(T)+ \int f\,d\mu
\end{align*}
Let $\epsilon\rightarrow 0$, then we have $P_{\text{pre}}(T,f)\leq
h_{\text{pre},\,\mu}^{(r)}(T)+ \int f\,d\mu$ and complete the proof
of Theorem \ref{theorem}.
\end{proof}

\begin{remark}
If $f=0$, then, without any additional assumption, Theorem
\ref{theorem} can be expressed as $h_{\text{pre}}^{(r)}(T)=\sup
\{h_{\text{pre},\,\mu}^{(r)}(T): \mu \in
\mathcal{M}_{\mathbf{P}}^1(\mathcal{E},T) \} $.  In \cite{Zhu}, the
variational principle for the pre-image topological entropy needs a
measurability condition which in most cases cannot be satisfied.
Thus Theorem \ref{theorem} can be regarded as a revised version of
Zhu's. On the other hand, if $\Omega$ consists of only one point,
that is, $\Omega=\{\omega\}$, then by Remark \ref{rem1}, Theorem
\ref{theorem} generalizes Zeng's deterministic variational principle
on the whole space $X$ \cite{Fanping} to any compact invariant
subset $E$.
\end{remark}

\section*{Acknowledgements}
The first author is supported by a grant from Postdoctoral Science
 Research Program of Jiangsu Province (0701049C).
 The second  author is partially supported by the National Natural Science Foundation of China (10571086)
 and National Basic Research Program of China (973 Program) (2007CB814800).

\end{document}